\documentclass[11pt]{amsart}
\usepackage[margin=1in]{geometry}

\usepackage{amssymb}
\usepackage{amsthm}
\usepackage{amsmath}
\usepackage{mathrsfs}
\usepackage{amsbsy}
\usepackage{bm}
\usepackage{hyperref}
\usepackage{tikz}
\usepackage{array}
\usepackage{svg}
\usepackage{enumerate}
\usepackage{bbm}
\usepackage{comment}
\usepackage{mathtools}
\usepackage{tabu}
\usepackage{makecell} 
\usepackage{colortbl}
\usepackage{xcolor}

\DeclareFontFamily{U}{mathx}{}
\DeclareFontShape{U}{mathx}{m}{n}{<-> mathx10}{}
\DeclareSymbolFont{mathx}{U}{mathx}{m}{n}
\DeclareMathAccent{\widecheck}{0}{mathx}{"71}

\definecolor{LightBlue}{rgb}{0,0.8,1} 
\definecolor{Green}{rgb}{0,0.863,0}
\definecolor{DarkGreen}{rgb}{0,0.5,0}
\definecolor{MildGreen}{rgb}{0,0.784,0}
\definecolor{Turquoise}{rgb}{0,0.68,0.38}
\definecolor{NormalGreen}{rgb}{0,0.8,0}
\definecolor{LightGreen}{rgb}{0,0.922,0}
\definecolor{Magenta}{rgb}{0.784,0,0.784}
\definecolor{Yellow}{rgb}{0.95,0.95,0}
\definecolor{lavender}{rgb}{0.4,0,1}
\definecolor{peach}{rgb}{1,0.43,0.39} 
\definecolor{DarkPink}{rgb}{1,0,0.45} 
\definecolor{NewBlue}{rgb}{0,0.3,0.8}
\definecolor{Teal}{rgb}{0,0.784,0.784}
\definecolor{Gold}{rgb}{0.929,0.784,0.392}

\hypersetup{colorlinks=true, citecolor=LightBlue, linkcolor=NewBlue,urlcolor=NewBlue}

\usepackage{hhline}
\allowdisplaybreaks
\usepackage[noadjust]{cite}

\usepackage{caption}
\usepackage[noabbrev,capitalise,nameinlink]{cleveref}
\usepackage{float}
\crefname{conjecture}{Conjecture}{Conjectures}

\newtheorem{theorem}{Theorem}[section]

\newtheorem{lemma}[theorem]{Lemma}

\theoremstyle{definition}

\usepackage{etoolbox}

\newcommand{\norm}[1]{\left\lVert #1 \right\rVert}

\newcommand{\D}{S}
\newcommand{\Var}{\mathrm{Var}} 
\newcommand{\p}{\mathfrak{p}}
\newcommand{\s}{\mathfrak{s}}

\newcommand{\dfn}[1]{\textcolor{blue}{\emph{#1}}}

\begin{document}

\title[]{The Minary Primitive of Computational Autopoiesis} 
\subjclass[2010]{}

\author[]{Daniel Connor}
\address[]{Autopoetic, New York, NY, USA}
\email{daniel@danielconnor.com} 

\author[]{Colin Defant}
\address[]{Department of Mathematics, Harvard University, Cambridge, MA 02138, USA}
\email{colindefant@gmail.com}

\begin{abstract}
We introduce \emph{Minary}, a computational framework designed as a candidate for the first formally provable autopoietic primitive. Minary represents interacting probabilistic events as multi-dimensional vectors and combines them via linear superposition rather than multiplicative scalar operations, thereby preserving uncertainty and enabling constructive and destructive interference in the range $[-1,1]$. A fixed set of ``perspectives'' evaluates ``semantic dimensions'' according to hidden competencies, and their interactions drive two discrete-time stochastic processes. We model this system as an iterated random affine map and use the theory of iterated random functions to prove that it converges in distribution to a unique stationary law; we moreover obtain an explicit closed form for the limiting expectation in terms of row, column, and global averages of the competency matrix. We then derive exact formulas for the mean and variance of the normalized consensus conditioned on the activation of a given semantic dimension, revealing how consensus depends on competency structure rather than raw input signals. Finally, we argue that Minary is organizationally closed yet operationally open in the sense of Maturana and Varela, and we discuss implications for building self-maintaining, distributed, and parallelizable computational systems that house a uniquely subjective notion of identity.
\end{abstract}

\maketitle

\section{Introduction}\label{sec:intro} 
The field of computer science has faced long-standing challenges in representing distributed probabilistic systems. The textbook answer to reconciling interacting probabilistic events has been to reach for non-linear multiplicative calculations of scalar values in the unit interval $[0,1]$. The effect of using multiplication with unit interval scalars is that the multiplied values approach zero, amplify noise, and fully collapses to $0$ in cases where any participant contributes a factor of $0$.

A potentially more robust alternative to using scalars throughout is to compute products with probability density functions, collapsing the result to a unit interval scalar at the end using methods such as computing the mean, mode, or median.

Bayes defined the core multiplicative process as Posterior $\propto$ Prior $\times$ Likelihood \cite{Bayes}.
In any case, multiplicative methods with scalar collapse are specifically intended to \textit{reduce} uncertainty.

The disclosed Minary computational framework uses a fundamentally different philosophy that is intended to \textit{preserve} uncertainty. The Minary framework combines interdependent probabilistic events, contributed by perspectives, represented end-to-end as multi-dimensional vectors. A simple linear transformation enables the interactions of these vectors to be represented with components in the range $[-1,1]$. The use of signed components enables the use of wave-like superposition to create constructive and destructive interference patterns. Unlike multiplicative methods, the superposition is linear.

The Minary computational framework consumes vectors, computes with vectors, and produces vectors. The linear property of superposition enables symmetry that provably preserves information, while vectors contribute a high-fidelity format where the information of uncertainty or belief can be represented with precision across arbitrary dimensions. Additionally, the commutative and associative properties of the superposition confer computational flexibility in terms of latency and parallelism.

The Minary computational framework is a dynamical system and cybernetic feedback mechanism \cite{Ashby} backed by learning stored in a semantic topology data structure. Think of it as a primitive for collective belief. The input signal functions as the perturbation driving emergence, but is demonstrated to be completely canceled out from the learning signals, thus creating a totally closed and self-referential learning loop. As a result of its vector and linear properties, even in a closed system, the feedback of the Minary framework does not drive the system towards collapse but instead maintains coherence through information preservation, something akin to the physical properties of Newtonian energy conservation.

Functionally closed but operationally open, Minary, we argue, is the first formal, provable model of an autopoietic computational primitive.

The Minary computational framework is usable in both stochastic and deterministic forms; this article explores the properties of the stochastic variant.  

\section{Definitions}

\subsection{Autopoiesis} The term \emph{autopoiesis} was first introduced in 1972 by biologists Humberto Maturana and Francisco Varela \cite{MaturanaVarela1972,MaturanaVarela1980}, who described living cells as self-creating machines. Building on the framework of general systems theory \cite{vonBertalanffy}, which characterized living systems as open systems maintaining themselves through continuous exchange with their environment, Maturana and Varela proposed a more specific organizational criterion. This naturally led to subsequent attempts in computer science to create artificial machines that meet their formal criteria, and attempts to model social dynamics using autopoietic methods \cite{Luhmann}. In their own words: ``An autopoietic machine is a machine organized (defined as a unity) as a network of processes of production (transformation and destruction) of components which: (i)~through their interactions and transformations continuously regenerate and realize the network of processes (relations) that produced them; and (ii)~constitute it (the machine) as a concrete unity in space in which they (the components) exist by specifying the topological domain of its realization as such a network.''

To qualify as autopoietic, a machine must continuously regenerate its own structure (be organizationally closed) while also responding to its environment (be operationally open).

\subsection{Allopoiesis} In contrast and mutually exclusive to autopoiesis, allopoietic machines are organized to produce something other than themselves (e.g., a car factory produces cars, not more factories). Their function is defined by external factors rather than self-referential maintenance. Most traditional computational processes are allopoietic. \cite{MaturanaVarelaUribe1974}

\section{Related Works}

\subsection{Bayesian Networks} First formally introduced in 1985 by Judea Pearl, Bayesian Networks form a foundational technique for probabilistic reasoning \cite{Pearl}. An applied implementation of Bayes' Theorem, these networks are composed of directed acyclic graphs that represent causal relationships between events.

In this framework, the prior probability is multiplied with the likelihood. The network's graph structure allows the joint probability distribution of all variables to be factored into a product of local conditional probabilities. The purpose of deploying a Bayesian Network is to infer the most likely state for a given set of events; thus, they are explicitly designed to collapse uncertainty and converge on a specific outcome.

Fundamentally allopoietic, a Bayesian Network's structure is organized by an external designer and updated with new evidence. Its organization is not self-producing but is instead a static map used for externally-directed inference.

\subsection{Artificial Neural Networks} Rooted in statistical learning theory and computational neuroscience, Artificial Neural Networks (ANNs) represent the work of many individuals over the past century \cite{Rosenblatt,McCullochPitts,Hebb,MinskyPapert,Hopfield} and have become the dominant paradigm for artificial intelligence (Large Language Models, in particular).

In an ANN, input vectors are fed through layers of computational units (``neurons''). Each unit applies a weighted, multiplicative sum to its inputs, which is then passed through a non-linear activation function. Training is typically guided by backpropagation \cite{RumelhartHintonWilliams}, a process that produces an error signal by comparing outputs against an external ``ground truth.'' This error signal is then used to update the model's weights, thus increasing the likelihood of the model producing output aligned with the ground truth signal in the future.

Fundamentally allopoietic, ANNs are externally directed systems whose organization is sculpted by an external objective and whose outputs are distinct from themselves.

\subsection{Vector Symbolic Architectures} A close cousin to Minary, Vector Symbolic Architectures (VSA) \cite{Kanerva} or Hyperdimensional Computing (HDC) \cite{Gayler} leverage superposition to bundle semantics as vectors in high-dimensional space where data is addressable by the encoded semantics. The resulting semantic topology may be queried in several ways, typically through distance-based similarity search (such as finding the $k$-Nearest Neighbors) to find the closest matches to a cue vector.

Fundamentally allopoietic, VSA and HDC architectures are populated by external sources of truth and do not produce themselves.

\subsection{The Autopoietic Gap} The allopoietic nature of virtually all dominant computational paradigms \cite{McMullin} leaves a gap for a novel autopoietic primitive for use in constructing self-directed systems. We position Minary as a candidate to fill that gap.

\section{Mathematical Formalism} 
In this section, we provide the rigorous mathematical definitions required to describe the stochastic processes modeling our consensus mechanism in the Minary framework. 

Let $\mathbb R^{a\times b}$ denote the set of real $a\times b$ matrices. For each positive integer $b$, let $[b]$ denote the set $\{1,2,\ldots,b\}$. 

We fix a set $\{\mathfrak p_1,\ldots,\mathfrak p_n\}$ of $n$ \dfn{perspectives} and a set $\{\mathfrak s_1,\ldots,\mathfrak s_m\}$ of $m$ \dfn{semantic dimensions}. Each perspective $\mathfrak p_i$ has a certain competency $C_{i,j}\in[0,1]$ for evaluating the semantic dimension $\mathfrak s_j$; the \dfn{competency matrix} is the matrix $C\in\mathbb R^{n\times m}$ whose entry in row $i$ and column $j$ is $C_{i,j}$. We also fix an integer $k\in[m]$, a \dfn{step size}\footnote{There is not much harm in taking $\alpha$ to be any number in $(0,1)$, but assuming it is less than $2/3$ will simplify one of our proofs. In specific examples, we will take $\alpha$ much less than $2/3$ (say around $0.02$).} $\alpha\in(0,2/3)$, and a probability measure $\mu$ on $[0,1]$. In examples, we will take $\mu$ to be the uniform measure on $[0,1]$.

At each time step $t$, we choose a $k$-element set $\D^{(t)}\subseteq[m]$ uniformly at random. The semantic dimensions $\mathfrak s_j$ for $j\in \D^{(t)}$ are the \dfn{active dimensions} at time $t$. For each $j\in \D^{(t)}$, we sample a random \dfn{signal} $x_j^{(t)}\in[0,1]$ from the probability distribution $\mu$. The signals chosen for different active dimensions are independent of each other. Each perspective $\mathfrak p_i$ then generates the \dfn{raw response} 
\begin{equation}\label{eq:r}
r_{i,j}^{(t)}=x_{j}^{(t)}-C_{i,j},
\end{equation} 
which is then adjusted using the exponential moving average to form the \dfn{adjusted response} 
\begin{equation}\label{eq:R}
R_{i,j}^{(t)}=r_{i,j}^{(t)}+\Delta_{i,j}^{(t-1)}.
\end{equation} 
At this point, the perspective $\mathfrak p_i$ has different adjusted responses for the different active dimensions. We consolidate this information into the single \dfn{average adjusted response} 
\begin{equation}\label{eq:R2}
R_{i}^{(t)}=\frac{1}{k}\sum_{j\in \D^{(t)}}R_{i,j}^{(t)}. 
\end{equation} 
This allows us to compute the consensus value 
\begin{equation}\label{eq:G}
G^{(t)}=\sum_{i=1}^nR_{i}^{(t)}. 
\end{equation} 
For all $i\in[n]$ and $j\in[m]$, we now update the exponential moving average by setting 
\begin{equation}\label{eq:Delta}
\Delta_{i,j}^{(t)}=\begin{cases} \alpha d_{i}^{(t)}+(1-\alpha)\Delta_{i,j}^{(t-1)} & \mbox{if }j\in \D^{(t)} \\  \Delta_{i,j}^{(t-1)} & \mbox{if }j\not\in \D^{(t)},  \end{cases} 
\end{equation} 
where 
\begin{equation}\label{eq:d}
d_{i}^{(t)}=\frac{1}{n}G^{(t)}-R_i^{(t)}. 
\end{equation}

In summary, there are two sources of randomness driving the processes $(\Delta^{(t)})_{t\geq 1}$ and $(G^{(t)})_{t\geq 1}$. One is the choice of the uniformly random $k$-element subset $\D^{(t)}\subseteq[m]$ at each time $t$; the other is the random signals $x_j^{(t)}$ (for $j\in \D^{(t)}$) at each time $t$. 

\section{Exponential Moving Average Limits}
Recall that we have a fixed competency matrix $C\in\mathbb R^{n\times m}$ with entries in $[0,1]$. For $i\in[n]$ and $j\in[m]$, let \[\overline C_{\cdot,j}=\frac{1}{n}\sum_{r=1}^n C_{r,j}\quad\text{and}\quad \overline C_{i,\cdot}=\frac{1}{m}\sum_{r=1}^m C_{i,r}.\] Let us also write 
\[\overline{\overline C}=\frac{1}{mn}\sum_{i=1}^n\sum_{j=1}^m C_{i,j}\] for the average of all competencies. 

Our goal in this section is to prove the following theorem regarding the exponential moving average process $(\Delta_t)_{t\geq 0}$. 

\begin{theorem}\label{thm:convergence} 
As $t\to\infty$, the random matrix $\Delta^{(t)}$ converges in distribution. Moreover, for all $i\in[n]$ and $j\in[m]$, we have 
\[\lim_{t\to\infty}\mathbb E\left[\Delta_{i,j}^{(t)}\right]=\left(\frac{1}{2}-\eta_{m,k}\right)\left(\overline C_{i,\cdot}-\overline{\overline C}\right)+\eta_{m,k}\left(C_{i,j}-\overline C_{\cdot,j}\right),\] 
where 
\[\eta_{m,k}=\frac{m-k}{k(m-1)+m-k}.\] 
\end{theorem} 

Let us use the notation \[\overline\Delta_{\cdot,j}^{(t)}=\frac{1}{n}\sum_{i=1}^{n}\Delta_{i,j}^{(t)}.\] By combining \eqref{eq:r}, \eqref{eq:R}, \eqref{eq:R2}, and \eqref{eq:G}, we find that \begin{align}
\nonumber \frac{1}{n}G^{(t)}&=\frac{1}{n}\sum_{i=1}^nR_i^{(t)} \\ 
\nonumber &=\frac{1}{k}\sum_{j\in \D^{(t)}}\left(\left(\frac{1}{n}\sum_{i=1}^n x_{j}^{(t)}\right)-\overline C_{\cdot,j}+\overline\Delta_{\cdot,j}^{(t-1)}\right) \\ 
\label{eq:G_formula}&=\frac{1}{k}\sum_{j\in \D^{(t)}}\left(x_{j}^{(t)}-\overline C_{\cdot,j}+\overline\Delta_{\cdot,j}^{(t-1)}\right).  
\end{align}
Therefore,
\begin{align*}
d_{i}^{(t)}&=\frac{1}{n}G^{(t)}-R_i^{(t)} \\ 
&=\frac{1}{k}\sum_{j\in \D^{(t)}}\left(x_{j}^{(t)}-\overline C_{\cdot,j}+\overline\Delta_{\cdot,j}^{(t-1)}\right)-\frac{1}{k}\sum_{j\in \D^{(t)}}\left(x_j^{(t)}-C_{i,j}+\Delta_{i,j}^{(t-1)}\right). 
\end{align*}
The terms with the stimuli $x_j^{(t)}$ cancel, so we are left with 
\begin{equation}\label{eq:d_final}
d_{i}^{(t)}=\frac{1}{k}\sum_{j\in \D^{(t)}}\left(C_{i,j}-\overline C_{\cdot,j}+\overline\Delta_{\cdot,j}^{(t-1)}-\Delta_{i,j}^{(t-1)}\right).
\end{equation} 
This implies that the only randomness influencing the transition from $\Delta^{(t-1)}$ to $\Delta^{(t)}$ is the choice of $\D^{(t)}$ (and not the signals $x_j^{(t)}$). 

We wish to represent the process $(\Delta^{(t)})_{t\geq 0}$ as a Markov chain driven by a random affine map; this will allow us to employ known results from the theory of such Markov chains in order to prove \cref{thm:convergence}. To this end, let $I_\ell\in\mathbb R^{\ell\times \ell}$ denote the $\ell\times \ell$ identity matrix, and let $J_\ell\in\mathbb R^{\ell\times \ell}$ be the 
$\ell\times \ell$ matrix whose entries are all $1$. Let $\overline J_\ell=\frac{1}{\ell}J_\ell$. We denote the transpose of a matrix $M$ by $M^\top$. Let $C^{\mathrm{dev}}\in\mathbb R^{n\times m}$ be the matrix defined by 
\begin{equation}\label{eq:Cdev}
C_{i,j}^{\mathrm{dev}}=C_{i,j}-\overline C_{\cdot,j}
\end{equation} 

For each set $S\subseteq[m]$, let $\delta^{S}\in\mathbb R^{m\times 1}$ be the column indicator vector of $S$ (so $\delta_j^S=1$ and $\delta_{j'}^S=0$ for all $j\in S$ and $j'\not\in S$). Let us also write $D^{S}\in\mathbb R^{m\times m}$ for the diagonal matrix whose $j$-th diagonal entry is $\delta_j^S$. Define the linear map $A^{S}\colon\mathbb R^{n\times m}\to \mathbb R^{n\times m}$ by 
\begin{equation}
A^{S}(M)=M(I_m-\alpha D^{S})+\frac{\alpha}{k}\left(\overline J_n-I_n\right)M\delta^{S}(\delta^{S})^\top
\end{equation} 
and the matrix $B^{S}\in\mathbb R^{n\times m}$ by 
\begin{equation}
B^{S}(M)=\frac{\alpha}{k}C^{\mathrm{dev}}\delta^{S}(\delta^{S})^\top. 
\end{equation} 
We obtain an affine map $\Phi^{S}\colon\mathbb R^{n\times m}\to\mathbb R^{n\times m}$ defined by \[\Phi^{S}(M)=A^{S}(M)+B^{S}.\]  
Combining \eqref{eq:Delta}, \eqref{eq:d_final}, and \eqref{eq:Cdev} with some elementary linear algebra yields the identity 
\begin{equation}\label{eq:DeltaPhi}
\Delta^{(t)}=A^{S^{(t)}}(\Delta^{(t-1)})+B^{S^{(t)}}=\Phi^{S^{(t)}}(\Delta^{(t-1)}).
\end{equation} 

There is a natural inner product $\langle\cdot,\cdot\rangle$ on $\mathbb R^{\ell\times \ell'}$ given by \[\langle M,M'\rangle=\mathrm{Tr}(M^\top M')=\sum_{i=1}^\ell\sum_{j=1}^{\ell'} M_{i,j}M'_{i,j},\] where $\mathrm{Tr}$ denotes trace. This induces the \dfn{Frobenius norm} on $\mathbb R^{\ell\times \ell'}$ given by \[\norm{M}=\langle M,M\rangle^{1/2},\] which makes $\mathbb R^{\ell\times \ell'}$ into a metric space. The \dfn{Lipschitz constant} of a map $F\colon\mathbb R^{\ell\times \ell'}\to\mathbb R^{\ell\times \ell'}$ is \[\mathrm{Lip}(F)=\sup_{M,M'\in\mathbb R^{\ell\times \ell'}}\frac{\norm{F(M)-F(M')}}{\norm{M-M'}}.\] 

The next lemma is the main technical ingredient needed to prove \cref{thm:convergence}. 

\begin{lemma}\label{lem:Lip} 
Fix an integer $b\geq m/k$. Let $S_1,\ldots,S_b$ be independent $k$-element subsets of $[m]$, each chosen uniformly at random, and let $\Psi=\Phi^{S_b}\circ\cdots\circ\Phi^{S_1}$. We have 
\[\mathbb E\left[\log(\mathrm{Lip}(\Psi))\right]<0.\] 
\end{lemma} 
\begin{proof}
Consider the subspaces 
\[V=\{M\in\mathbb R^{n\times m}:(I_n-\overline J_n)M=0\}\quad\text{and}\quad V^\perp=\{M\in\mathbb R^{n\times m}:\overline J_nM=0\},\] which are orthogonal complements of each other. For $S\subseteq[m]$, let \[Q^S=I_n-\alpha D^S-\frac{\alpha}{k}\delta^S(\delta^S)^\top.\] 

For each subset $S\subseteq[m]$ and all matrices $M\in V$ and $M'\in V^\perp$, we have 
\[A^S(M)=M(I_m-\alpha D^S)\in V\quad\text{and}\quad A^S(M')=M'Q^S\in V^\perp.\] This shows that $V$ and $V^\perp$ are both invariant under $A^S$. Let $A^S\vert_V$ and $A^S\vert_{V^\perp}$ be the restrictions of $A^S$ to $V$ and $V^\perp$, respectively. 

For $M\in V$, we have \[\norm{A^S(M)}=\norm{M(I_m-\alpha D^{S})}\leq\norm{M};\] this shows that $\mathrm{Lip}(A^S\vert_V)\leq 1$. 
For $M'\in V^\perp$, since the matrix $Q^S$ is symmetric, we have 
\[
\norm{A^{S}(M')}^2=\norm{MQ^{S}}^2=\mathrm{Tr}(Q^SM^\top M Q^S)=\mathrm{Tr}(M^\top M (Q^S)^2). 
\]
A straightforward computation shows that 
\[(Q^S)^2=I_m-\alpha(2-\alpha)D^S+\frac{\alpha}{k}(3\alpha-2)\delta^S(\delta^S)^\top.\] Hence, 
\begin{align*}
\norm{A^S(M')}^2&=\mathrm{Tr}\left((M')^\top M' \left(I_m-\alpha(2-\alpha)D^S+\frac{\alpha}{k}(3\alpha-2)\delta^S(\delta^S)^\top\right)\right) \\ 
&=\norm{M'}^2-\alpha(2-\alpha)\norm{M'D^S}^2+\frac{\alpha}{k}(3\alpha-2)\norm{M'\delta^S}^2. 
\end{align*} 
We have assumed that $0<\alpha<2/3$, so 
\[\norm{A^S(M')}^2\leq\norm{M'}^2-\alpha(2-\alpha)\norm{M'D^S}^2\leq\norm{M'}^2.\] This shows that $\mathrm{Lip}(A^S\vert_{V^\perp})\leq 1$. Consequently, $\mathrm{Lip}(A^S)=\max\{\mathrm{Lip}(A^S\vert_V),\mathrm{Lip}(A^S\vert_{V^\perp})\}\leq 1$. 

It follows from the preceding paragraph that 
\[
\mathrm{Lip}(\Psi)=\mathrm{Lip}(A^{S_b}\cdots A^{S_1})\leq\mathrm{Lip}(A^{S_b})\cdots\mathrm{Lip}(A^{S_1})\leq 1. 
\]
Therefore, to prove that $\mathbb E[\log(\mathrm{Lip}(\Psi))]<0$, we just need to show that $\mathrm{Lip}(\Psi)<1$ with positive probability. Because $b>m/k$, the probability that $S_1\cup\cdots\cup S_b=[m]$ is positive. Hence, it suffices to show that $\mathrm{Lip}(\Psi)<1$ if $S_1\cup\cdots\cup S_b=[m]$. 

Suppose $S_1\cup\cdots\cup S_b=[m]$. We have $\mathrm{Lip}(\Psi)=\mathrm{Lip}(A^{S_b}\cdots A^{S_1})$. Therefore, we just need to show that $\norm{A^{S_b}\cdots A^{S_1}(M)}<\norm{M}$ and $\norm{A^{S_b}\cdots A^{S_1}(M')}<\norm{M'}$ for all nonzero $M\in V$ and $M'\in V^\perp$. For the first inequality, we have \[
\norm{A^{S_b}\cdots A^{S_1}(M)}=\norm{M(I_m-\alpha D^{S_1})\cdots (I_m-\alpha D^{S_b})}<\norm{M}.
\] For the second inequality, choose $j^*\in[m]$ such that the $j^*$-th column of $M'$ has at least one nonzero entry, and let $r$ be the smallest element of $[b]$ such that $j^*\in S_r$. Note that the $j^*$-th column of $M'$ is the same as the $j^*$-th column of the matrix $M''=M'Q^{S_1}\cdots Q^{S_{r-1}}=A^{S_{r-1}}\cdots A^{S_1}(M')$. Since $j^*\in S_r$ and $0<\alpha<2/3$, we have 
\begin{align*}
\norm{A^{S_r}(M'')}^2=\norm{M''}^2-\alpha(2-\alpha)\norm{M''D^{S_r}}^2+\frac{\alpha}{k}(3\alpha-2)\norm{M''\delta^{S_r}}^2<\norm{M''}^2. 
\end{align*}
Therefore, 
\begin{align*}
\norm{A^{S_b}\cdots A^{S_1}(M')}&=\norm{A^{S_b}\cdots A^{S_{r+1}}(A^{S_r}(M''))} \\ 
&\leq\norm{A^{S_r}(M'')} \\ 
&<\norm{M''} \\ 
&=\norm{A^{S_{r-1}}\cdots A^{S_1}(M')} \\ 
&\leq\norm{M'}, 
\end{align*}
as desired. 
\end{proof} 

We will appeal to the following special case of a result due to Diaconis and Freedman. 

\begin{theorem}[{\cite[Theorem~1]{DiaconisFreedman}}]\label{thm:DF}
Let $\mathcal X$ be a separable metric space with metric $\varrho$. Let $\Theta$ be a finite set, and for each $\theta\in \Theta$, suppose we have a function $f_\theta\colon \mathcal X\to X$ and a real number $K_\theta\geq 0$ such that $\varrho(f_\theta(x),f_\theta(x'))\leq K_\theta\varrho(x,x')$ for all $x,x'\in\mathcal X$. Let $\nu$ be a probability measure on $\Theta$, and let $\theta_1,\theta_2,\ldots$ be an i.i.d.\ sequence of elements of $\Theta$ with distribution $\nu$. Let $X_0\in\mathcal X$, and for each integer $t\geq 1$, let $X_t=(f_{\theta_{t}}\circ f_{\theta_{t-1}}\circ\cdots\circ f_{\theta_1})(X_0)$. Assume that $\sum_{\theta\in\Theta} \log(K_\theta)\nu(\theta)<0$. Then the Markov chain $(X_t)_{t\geq 0}$ has a unique stationary distribution $\pi$, and the law of $X_t$ converges to $\pi$ exponentially. 
\end{theorem} 

We can now combine \cref{lem:Lip,thm:DF} to prove \cref{thm:convergence}. 

\begin{proof}[Proof of \cref{thm:convergence}]
Let $\mathcal X=\mathbb R^{n\times m}$; this is a separable metric space with metric $\varrho$ given by $\varrho(M,M')=\norm{M-M'}$. Fix an integer $b\geq m/k$, and take $\Theta$ be the collection of $b$-tuples of $k$-element subsets of $[m]$. Let $\nu$ be the uniform distribution on $\Theta$. For each tuple $\theta=(S_1,\ldots,S_b)\in\Theta$, let $f_\theta=\Phi^{S_b}\cdots\Phi^{S_1}$, and let $K_\theta=\mathrm{Lip}(f_\theta)$. Let $X_t=\Delta^{(bt)}$; it follows from \eqref{eq:DeltaPhi} that $X_t=(f_{\theta_t}\circ f_{\theta_{t-1}}\circ\cdots\circ f_{\theta_1})(X_0)$. \cref{lem:Lip} tells us that 
$\sum_{\theta\in\Theta}\log(K_\theta)\nu(\theta)<0$. All of the hypotheses of \cref{thm:DF} are satisfied, so we conclude that the Markov chain $(X_t)_{t\geq 0}$ has a unique stationary distribution $\pi$ and that the law of $X_t$ converges to $\pi$ exponentially. Since $X_t=\Delta^{(bt)}$, this proves the first statement of the theorem. 

We now know that the limit $E=\lim_{t\to\infty}\mathbb E\left[\Delta^{(t)}\right]\in\mathbb R^{n\times m}$ exists. Let $I$ denote the identity map on $\mathbb R^{n\times m}$. It follows from \eqref{eq:DeltaPhi} that $E$ satisfies the equation 
\[(I-\mathbb E[A^S])E=\mathbb E[B^S],\] where the expected values are computed by choosing $S$ uniformly at random from the collection of $k$-element subsets $[m]$. To see that this equation has a unique solution, note that, by the proof of \cref{lem:Lip}, the linear map $\mathbb E[A^S]\colon\mathbb R^{n\times m}\to\mathbb R^{n\times m}$ has a Lipschitz constant strictly less than $1$, implying that it has no nonzero fixed points. It follows that $I-\mathbb E[A^S]$ is invertible, so we must have $E=(I-\mathbb E[A^S])^{-1}\mathbb E[B^S]$. 

Let \[p_1=\frac{k}{m}\quad\text{and}\quad p_2=\frac{k(k-1)}{m(m-1)}.\] Let 
\[W=\mathbb E[\delta^S(\delta^S)^\top]=p_2J_m+(p_1-p_2)I_m.\] 
Consider the matrices $R,C^{\mathrm{dev}}\in\mathbb R^{n\times m}$ defined by \[R_{i,j}=\overline C_{i,\cdot}-\overline{\overline C}\quad\text{and}\quad C^{\mathrm{dev}}_{i,j}=\overline C_{i,j}-\overline C_{\cdot,j}.\] Let $U=\left(\frac{1}{2}-\eta_{m,k}\right)R+\eta_{m,k}C^{\mathrm{dev}}$, where 
\[\eta_{m,k}=\frac{m-k}{k(m-1)+m-k}.\] 
We have \[\overline J_n R=\overline J_nC^{\mathrm{dev}}=0,\quad RJ_m=mR,\quad\text{and} \quad C^{\mathrm{dev}}J_m=mC^{\mathrm{dev}}.\] A straightforward computation shows that 
\[RW=\frac{k^2}{m}R\quad\text{and}\quad C^{\mathrm{dev}}W=(p_1-p_2)C^{\mathrm{dev}}+p_2mR.\] From this, we compute that 
\[U=(1-p_1\alpha)U+\frac{\alpha}{k}(\overline{J}_n-I_n)UW+\frac{\alpha}{k}C^{\mathrm{dev}}W=\mathbb E[A^S]U+\mathbb E[B^S].\] It follows that $E=U$, as desired. 
\end{proof} 

\section{The Consensus Distribution} 

Let $\overline \mu$ and $\sigma$ be the mean and standard deviation, respectively, of the probability distribution $\mu$. (If $\mu$ is the uniform distribution on $[0,1]$, then $\overline\mu=1/2$ and $\sigma=1/\sqrt{12}$.) 

For each $j\in[m]$, we are interested in the normalized consensus value \[\overline G^{(t)}=\frac{1}{n}G^{(t)}\] conditioned on the event that $j\in\D^{(t)}$. We have the following theorem. 

\begin{theorem}
Fix $j\in[m]$. Let $\widehat C_j=\frac{1}{m-1}\sum_{r\in[m]\setminus\{j\}}\overline C_{\cdot,r}$. The conditional expectation of $\overline G^{(t)}$ given that $j\in\D^{(t)}$ is given by 
\[\mathbb E\left[\overline G^{(t)}\mid j\in\D^{(t)}\right]=\overline\mu-\frac{1}{k}\left(\overline C_{\cdot,j}-(k-1)\widehat C_j\right).\] The conditional variance of $\overline G^{(t)}$ given that $j\in\D^{(t)}$ is given by 
\[\Var\left(\overline G^{(t)}\mid j\in\D^{(t)}\right)=\frac{1}{k}\sigma^2+\frac{(k-1)(m-k)}{k^2(m-1)(m-2)}\sum_{r\in[m]\setminus \{j\}}\left(\overline C_{\cdot,r}-\widehat C_j\right)^2.\]
\end{theorem}

\begin{proof}
It is immediate from \eqref{eq:G} and \eqref{eq:d} that $\sum_{i=1}^n d_i^{(t)}=0$. Therefore, 
\[\overline\Delta_{\cdot,j}^{(t)}=\begin{cases} (1-\alpha)\overline{\Delta}_{\cdot,j}^{(t-1)} & \mbox{if }j\in \D^{(t)} \\  \Delta_{i,j}^{(t-1)} & \mbox{if }j\not\in \D^{(t)}.  \end{cases} \]
Since $\Delta^{(0)}=0$, we must have $\overline\Delta_{\cdot,j}^{(t)}=0$ for all $t\geq 0$. Hence, if we condition on the event that $j\in\D^{(t)}$, then \eqref{eq:G_formula} tells us that 
\[\overline G^{(t)}=\frac{1}{k}\sum_{r\in \D^{(t)}}x_r^{(t)}-\frac{1}{k}\sum_{r\in\D^{(t)}\setminus \{j\}}\overline C_{\cdot,r}-\frac{1}{k}\overline C_{\cdot,j}.\] This immediately implies the desired formula for $\mathbb E[G^{(t)}\mid j\in\D^{(t)}]$. 

Since the signals $x_j^{(t)}$ are chosen independently at random from the distribution $\mu$, the variance of the random variable $\frac{1}{k}\sum_{j\in\D^{(t)}}x_j^{(t)}$ is $\frac{1}{k}\sigma^2$. The variance of the random variable $\frac{1}{k}\sum_{r\in\D^{(t)}\setminus\{j\}}\overline C_{\cdot,r}$ is 
\begin{align*}
&\hphantom{==}\frac{1}{k^2}{}\sum_{r\in[m]\setminus \{j\}}\overline C_{\cdot,r}^2\Var(\delta^{(\D^{(t)})}_r)+\frac{1}{k^2}\sum_{\substack{r,\ell\in[m]\\ r\neq\ell}}\overline C_{\cdot,r}\overline C_{\cdot,\ell}\mathrm{Cov}(\delta^{(\D^{(t)})}_r,\delta^{(\D^{(t)})}_\ell) \\ 
&=\frac{(k-1)(m-k)}{k^2(m-1)^2}\sum_{r\in[m]\setminus \{j\}}\overline C_{\cdot,r}^2-\frac{(k-1)(m-k)}{k^2(m-1)^2(m-2)}\sum_{\substack{r,\ell\in[m]\setminus\{j\}\\ r\neq\ell}}\overline C_{\cdot,r}\overline C_{\cdot,\ell} \\ 
&=\frac{(k-1)(m-k)}{k^2(m-1)^2(m-2)}\left((m-2)\sum_{r\in[m]\setminus \{j\}}\overline C_{\cdot,r}^2-\sum_{\substack{r,\ell\in[m]\setminus\{j\}\\ r\neq\ell}}\overline C_{\cdot,r}\overline C_{\cdot,\ell}\right) \\ 
&=\frac{(k-1)(m-k)}{k^2(m-1)^2(m-2)}\left((m-1)\sum_{r\in[m]\setminus \{j\}}\overline C_{\cdot,r}^2-\left((m-1)\widehat C_j\right)^2\right) \\ &=\frac{(k-1)(m-k)}{k^2(m-1)(m-2)}\left(\sum_{r\in[m]\setminus \{j\}}\overline C_{\cdot,r}^2+(m-1)\widehat C_j^2-2\widehat C_j\sum_{r\in[m]\setminus\{j\}}\overline C_{\cdot,r}\right) \\ &=\frac{(k-1)(m-k)}{k^2(m-1)(m-2)}\sum_{r\in[m]\setminus \{j\}}\left(\overline C_{\cdot,r}-\widehat C_j\right)^2.
\end{align*}
Since these two random variables are independent, the desired formula for $\Var(\overline G^{(t)}\mid j\in\D^{(t)})$ follows.  
\end{proof} 

\section{Worked Example}

To illustrate the mechanics of the Minary framework, a Python simulation of a stochastic Minary is available \cite{code} wherein we have prepared 5 perspectives along with 19 dimensions identified with certain competency labels. Each perspective has been assigned a competency value in $[0,1]$ for each dimension; together, these roughly create profiles that reflect their respective ``archetypes'' of the perspectives. This arrangement is large enough to produce rich dynamics while small enough to not be unwieldy.

This particular example has perspectives evaluate multiple semantic dimensions coupled together 3 at a time by responding with an average of their responses across all active dimensions. This reflects the idea that each iteration represents one holistic unit that require all three competencies at once.

\subsection{Setup}

We consider the following system:
\begin{itemize}
    \item $n = 5$ perspectives: The True Artist ($\p_1$), The Executive Director ($\p_2$), The Technician ($\p_3$), The Critic ($\p_4$), and The Fan ($\p_5$)
    \item $m = 19$ semantic dimensions: ``3d modeling'' ($\s_1$), ``anatomy'' ($\s_2$), ``artwork similarity'' ($\s_3$), ``audience relevance'' ($\s_4$), ``brand voice'' ($\s_5$), ``character design'' ($\s_6$), ``color grading'' ($\s_7$), ``color scheme'' ($\s_8$), ``costume design'' ($\s_9$), ``fashion trend'' ($\s_{10}$), ``illustration'' ($\s_{11}$), ``information redaction'' ($\s_{12}$), ``interior design'' ($\s_{13}$), ``modern art'' ($\s_{14}$), ``photographic composition'' ($\s_{15}$), ``physics'' ($\s_{16}$), ``sentiment'' ($\s_{17}$), ``usability'' ($\s_{18}$), ``visual ad'' ($\s_{19}$)
    \item $k = 3$ active dimensions at each time step (dimensions are coupled when perspectives average across all active dimensions)
    \item $\alpha = 0.02$ (step size for exponential moving average).
\end{itemize}

We will work through an iteration in which the dimensions ``brand voice'', ``modern art'', and ``physics'' are active. The competency matrix $C \in \mathbb{R}^{5 \times 3}$ is:
\[
C = \begin{bmatrix}
0.95 & 0.20 & 0.50 \\
0.70 & 0.97 & 0.30 \\
0.50 & 0.30 & 0.95 \\
0.80 & 0.87 & 0.10 \\
0.60 & 0.70 & 0.30
\end{bmatrix}
\]

We initialize $\Delta^{(0)} = 0$ (the zero matrix). At time $t = 1$, all three dimensions are active: $M^{(1)} = \{5,14,16\}$, and we draw signals from a uniform distribution on $[0, 1]$: $x^{(1)} = [0.6394, 0.0250, 0.2750]$.

\subsection{Step-by-Step Computation}

\subsubsection{Step 1: Raw Responses}

Using \cref{eq:r}, we compute $r_{i,j}^{(1)} = x_j^{(1)} - C_{i,j}$ for each perspective $i$ and dimension $j$:

\begin{align*}
\text{The True Artist:} \quad 
r_{1,1}^{(1)} &= 0.6394 - 0.95 = -0.3106 \\
r_{1,2}^{(1)} &= 0.0250 - 0.20 = -0.1750 \\
r_{1,3}^{(1)} &= 0.2750 - 0.50 = -0.2250
\end{align*}

\begin{align*}
\text{The Executive Director:} \quad 
r_{2,1}^{(1)} &= 0.6394 - 0.70 = -0.0606 \\
r_{2,2}^{(1)} &= 0.0250 - 0.97 = -0.9450 \\
r_{2,3}^{(1)} &= 0.2750 - 0.30 = -0.0250
\end{align*}

\begin{align*}
\text{The Technician:} \quad 
r_{3,1}^{(1)} &= 0.6394 - 0.50 = 0.1394 \\
r_{3,2}^{(1)} &= 0.0250 - 0.30 = -0.2750 \\
r_{3,3}^{(1)} &= 0.2750 - 0.95 = -0.6750
\end{align*}

\begin{align*}
\text{The Critic:} \quad 
r_{4,1}^{(1)} &= 0.6394 - 0.80 = -0.1606 \\
r_{4,2}^{(1)} &= 0.0250 - 0.87 = -0.8450 \\
r_{4,3}^{(1)} &= 0.2750 - 0.10 = 0.1750
\end{align*}

\begin{align*}
\text{The Fan:} \quad 
r_{5,1}^{(1)} &= 0.6394 - 0.60 = 0.0394 \\
r_{5,2}^{(1)} &= 0.0250 - 0.70 = -0.6750 \\
r_{5,3}^{(1)} &= 0.2750 - 0.30 = -0.0250
\end{align*}

\subsubsection{Step 2: Adjusted Responses}

Since $\Delta^{(0)} = 0$, the adjusted responses $R_{i,j}^{(1)} = r_{i,j}^{(1)}$ are identical to the raw responses.

\subsubsection{Step 3: Average Adjusted Responses}

Each perspective averages its adjusted responses across all active dimensions. This single average value is then used for all dimensions, creating a coupling between them:

\begin{align*}
R_1^{(1)} &= \frac{1}{3}(-0.3106 + (-0.1750) + (-0.2250)) = \frac{-0.7105}{3} = -0.2368 \\
R_2^{(1)} &= \frac{1}{3}(-0.0606 + (-0.9450) + (-0.0250)) = \frac{-1.0305}{3} = -0.3435 \\
R_3^{(1)} &= \frac{1}{3}(0.1394 + (-0.2750) + (-0.6750)) = \frac{-0.8105}{3} = -0.2702 \\
R_4^{(1)} &= \frac{1}{3}(-0.1606 + (-0.8450) + 0.1750) = \frac{-0.8305}{3} = -0.2768 \\
R_5^{(1)} &= \frac{1}{3}(0.0394 + (-0.6750) + (-0.0250)) = \frac{-0.6605}{3} = -0.2202
\end{align*}

\subsubsection{Step 4: Consensus (Superposition)}

The consensus is computed via superposition (linear summation):
\begin{align*}
G^{(1)} &= \sum_{i=1}^{5} R_i^{(1)} \\
&= -0.2368 + (-0.3435) + (-0.2702) + (-0.2768) + (-0.2202) \\
&= -1.3476
\end{align*}

\subsubsection{Step 5: Normalized Consensus and Learning Signals}

The normalized consensus is:
\[
\bar{G}^{(1)} = \frac{G^{(1)}}{n} = \frac{-1.3476}{5} = -0.2695
\]

The learning signal for each perspective (\cref{eq:d}) is:
\begin{align*}
d_1^{(1)} &= \bar{G}^{(1)} - R_1^{(1)} = -0.2695 - (-0.2368) = -0.0327 \\
d_2^{(1)} &= \bar{G}^{(1)} - R_2^{(1)} = -0.2695 - (-0.3435) = 0.0740 \\
d_3^{(1)} &= \bar{G}^{(1)} - R_3^{(1)} = -0.2695 - (-0.2702) = 0.0007 \\
d_4^{(1)} &= \bar{G}^{(1)} - R_4^{(1)} = -0.2695 - (-0.2768) = 0.0073 \\
d_5^{(1)} &= \bar{G}^{(1)} - R_5^{(1)} = -0.2695 - (-0.2202) = -0.0493
\end{align*}

\textbf{Verification:} $\sum_{i=1}^{5} d_i^{(1)} = -0.0327 + 0.0740 + 0.0007 + 0.0073 + (-0.0493) = 0$

\subsubsection{Step 6: Update Exponential Moving Average}

Since all perspectives use a single averaged response value across all dimensions, the learning signal $d_i^{(1)}$ is applied uniformly to all dimensions. Using $\alpha = 0.02$:

\begin{align*}
\text{The True Artist:} \quad \Delta_{1,j}^{(1)} &= 0.02 \times (-0.0327) = -0.000653 \text{ for } j \in \{1,2,3\} \\
\text{The Executive Director:} \quad \Delta_{2,j}^{(1)} &= 0.02 \times 0.0740 = 0.001480 \text{ for } j \in \{1,2,3\} \\
\text{The Technician:} \quad \Delta_{3,j}^{(1)} &= 0.02 \times 0.0007 = 0.000013 \text{ for } j \in \{1,2,3\} \\
\text{The Critic:} \quad \Delta_{4,j}^{(1)} &= 0.02 \times 0.0073 = 0.000147 \text{ for } j \in \{1,2,3\} \\
\text{The Fan:} \quad \Delta_{5,j}^{(1)} &= 0.02 \times (-0.0493) = -0.000987 \text{ for } j \in \{1,2,3\}
\end{align*}

The complete $\Delta^{(1)}$ matrix is:
\[
\Delta^{(1)} = \begin{bmatrix}
-0.000653 & -0.000653 & -0.000653 \\
0.001480 & 0.001480 & 0.001480 \\
0.000013 & 0.000013 & 0.000013 \\
0.000147 & 0.000147 & 0.000147 \\
-0.000987 & -0.000987 & -0.000987
\end{bmatrix}
\]

\subsection{Key Observations}

This worked example demonstrates the fundamental properties of the Minary framework:

\begin{enumerate}
    \item \textbf{Signal Cancellation}: The input signals $x^{(1)} = [0.6394, 0.0250, 0.2750]$ appear in the raw responses but completely cancel out in the learning signals. As shown in \cref{eq:d_final}, only competency-based differences remain in $d_i^{(1)}$, demonstrating functional closure.
    
    \item \textbf{Information Conservation}: $\sum_{i=1}^{5} d_i^{(1)} = 0$ exactly, confirming that the system preserves information through linear superposition rather than destroying it through multiplicative collapse. This property enables the system to maintain coherence indefinitely without converging to a degenerate state.
    
    \item \textbf{Coupled Dimension Behavior}: Each perspective averages across all active dimensions and applies the same learning signal uniformly. This creates dimensional coupling where $\Delta_{i,1}^{(1)} = \Delta_{i,2}^{(1)} = \Delta_{i,3}^{(1)}$ for each perspective $i$. This models scenarios where evaluations are holistic units of work rather than dimension-specific sub-units of work. (Note that it is also possible to respond with a full vector that maintains dimensional independence, which maintains simple, orthogonal dynamics.)
    
    \item \textbf{Perspective-Specific Adaptation}: The Executive Director, whose averaged response was most negative relative to the normalized consensus ($-0.3435$ vs $-0.2695$), receives the largest positive adjustment ($+0.0740$). Conversely, The Fan receives a negative adjustment ($-0.0493$). This demonstrates how perspectives learn to align with collective behavior while maintaining their individual competency profiles.
    
    \item \textbf{Emergent Self-Reference}: The exponential moving average matrices $\Delta^{(t)}$ evolve over time, forming an emergent semantic topology. This topology represents each perspective's learned adjustments independently of external signals, creating a self-referential structure that defines the system's organizational closure.
\end{enumerate}

\begin{figure}[]
    \centering
    \includegraphics[width=0.75\linewidth]{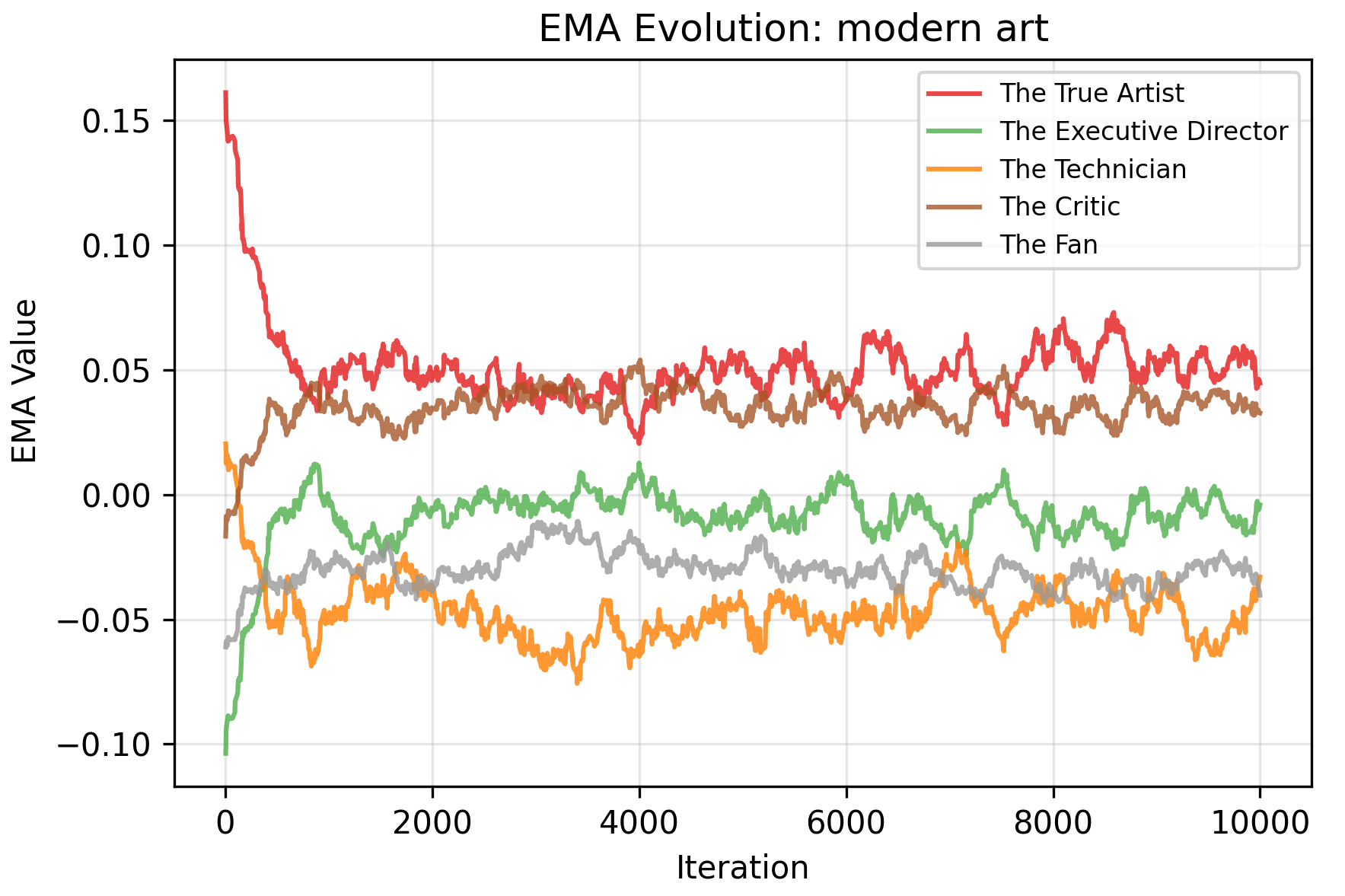}   
    \includegraphics[width=0.45\linewidth]{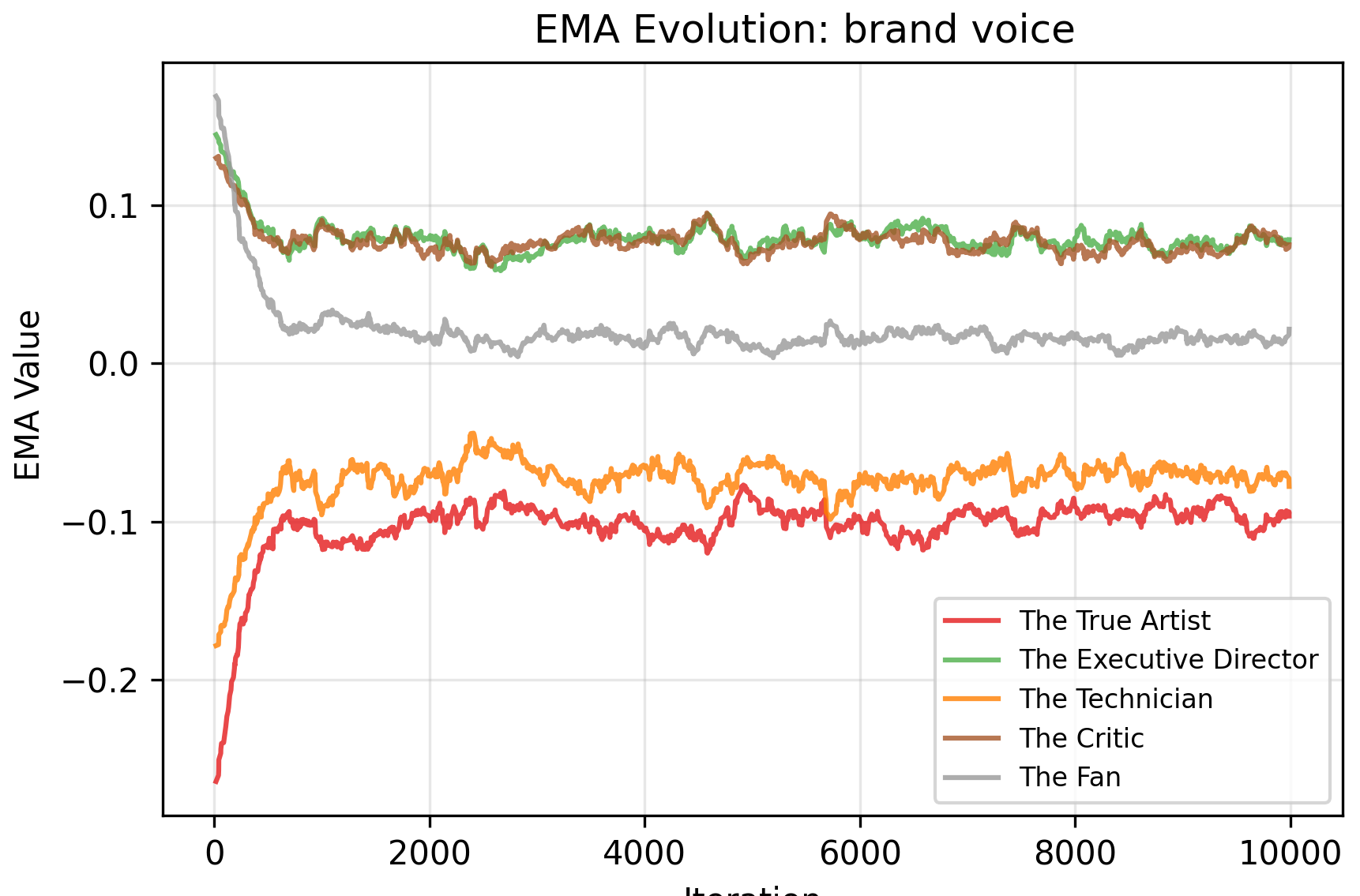}
    \includegraphics[width=0.45\linewidth]{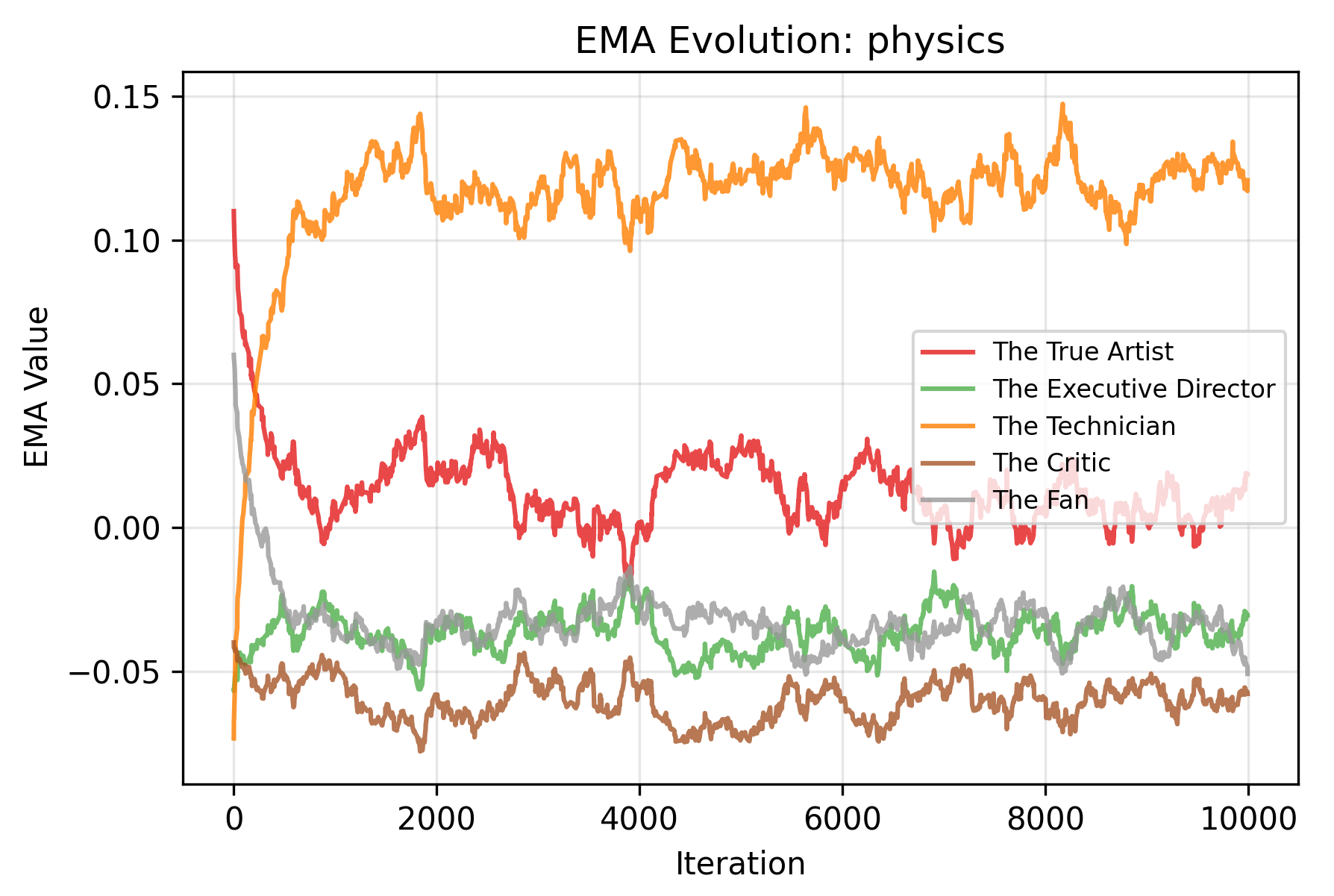}
\end{figure}

After 10000 such iterations with varying active dimensions, the system has developed a rich internal structure where the matrices $\Delta^{(t)}$ maintain the overall perspective archetypes in a dynamic way without becoming rigid or static.

\begin{figure}[]
    \centering
    \includegraphics[width=1\linewidth]{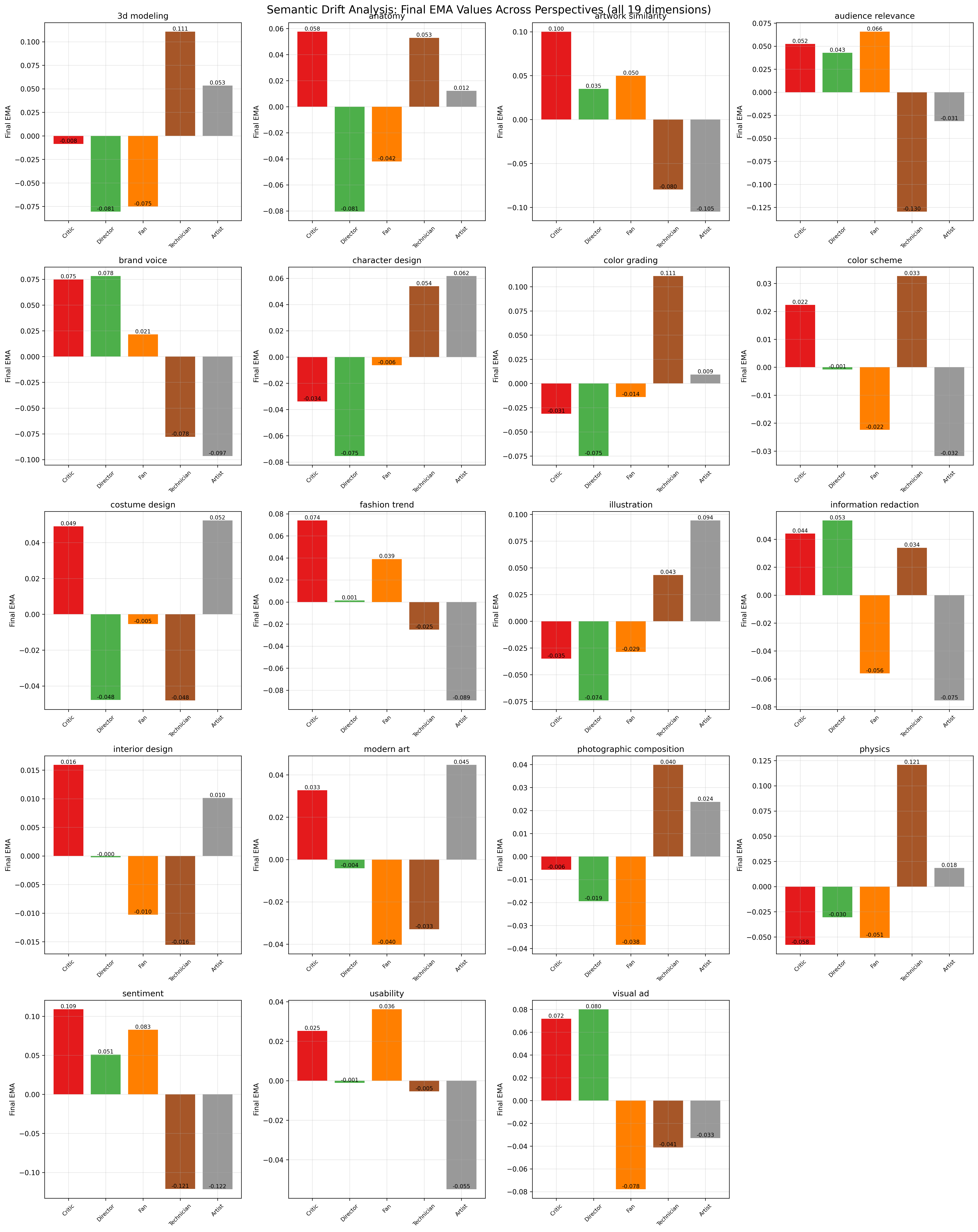}
\end{figure}

\subsection{Alternative Worked Example: Promotion of The Generalist.}

\subsubsection{Alternative Setup}

To clearly demonstrate a non-obvious emergent property of Minary, we define a minimalist configuration that incorporates a ``Generalist'' perspective.

\begin{itemize}
    \item $n = 3$ perspectives: The Specialist ($\p_1$), The Generalist ($\p_2$), The Anti-Specialist ($\p_3$)
    \item $m = 6$ semantic dimensions: $\s_1,\s_2,\ldots,\s_6$
    \item $k = 3$ active dimensions per iteration
    \item $\alpha = 0.02$
\end{itemize}

The competency matrix $C \in \mathbb{R}^{3 \times 6}$ is designed to highlight the phenomenon:

\[
C = \begin{pmatrix}
0.95 & 0.90 & 0.85 & 0.15 & 0.10 & 0.05 \\
0.50 & 0.50 & 0.50 & 0.50 & 0.50 & 0.50 \\
0.05 & 0.10 & 0.15 & 0.85 & 0.90 & 0.95
\end{pmatrix}
\]

The Specialist excels at dimensions $\s_1,\s_2,\s_3$ but is weak at $\s_4,\s_5,\s_6$. The Anti-Specialist has the opposite profile. The Generalist maintains $0.50$ competency across all dimensions.

\subsubsection{Initial Iterations Demonstrating Emergence}

\textbf{Iteration 1:} Active dimensions $S^{(1)} = \{1, 4, 5\}$ (mixing specialist strengths).

Signals: $x^{(1)} = [0.70, -, -, 0.30, 0.60, -]$

\textit{The Specialist's raw responses:}
\begin{align*}
r_{1,1}^{(1)} &= 0.70 - 0.95 = -0.25 \quad \text{(strong in } \s_1\text{)} \\
r_{1,4}^{(1)} &= 0.30 - 0.15 = 0.15 \quad \text{(weak in } \s_4\text{)} \\
r_{1,5}^{(1)} &= 0.60 - 0.10 = 0.50 \quad \text{(weak in } \s_5\text{)}
\end{align*}
Average: $R_1^{(1)} = \frac{-0.25 + 0.15 + 0.50}{3} = \mathbf{0.133}$

\textit{The Generalist's raw responses:}
\begin{align*}
r_{2,1}^{(1)} &= 0.70 - 0.50 = 0.20 \\
r_{2,4}^{(1)} &= 0.30 - 0.50 = -0.20 \\
r_{2,5}^{(1)} &= 0.60 - 0.50 = 0.10
\end{align*}
Average: $R_2^{(1)} = \frac{0.20 - 0.20 + 0.10}{3} = \mathbf{0.033}$

\textit{The Anti-Specialist's raw responses:}
\begin{align*}
r_{3,1}^{(1)} &= 0.70 - 0.05 = 0.65 \quad \text{(weak in } \s_1\text{)} \\
r_{3,4}^{(1)} &= 0.30 - 0.85 = -0.55 \quad \text{(strong in } \s_4\text{)} \\
r_{3,5}^{(1)} &= 0.60 - 0.90 = -0.30 \quad \text{(strong in } \s_5\text{)}
\end{align*}
Average: $R_3^{(1)} = \frac{0.65 - 0.55 - 0.30}{3} = \mathbf{-0.067}$

\textbf{Key Observation:} The Generalist's response $(0.033)$ is closest to zero, indicating the most balanced evaluation across mixed competencies.

The consensus: $G^{(1)} = 0.133 + 0.033 + (-0.067) = 0.099$

Normalized: $\bar{G}^{(1)} = \frac{0.099}{3} = 0.033$

Learning signals:
\begin{align*}
d_1^{(1)} &= 0.033 - 0.133 = -0.100 \quad \text{(Specialist penalized)} \\
d_2^{(1)} &= 0.033 - 0.033 = 0.000 \quad \text{(Generalist neutral)} \\
d_3^{(1)} &= 0.033 - (-0.067) = 0.100 \quad \text{(Anti-Specialist boosted)}
\end{align*}

\subsubsection{Evolution Over 1000 Iterations}

After 1000 iterations with randomly selected dimension triplets, the system exhibits remarkable behavior.

\textbf{Convergent Behavior for Mixed Dimensions:}
When active dimensions mix specialist strengths (e.g., $\{1,4,5\}$ or $\{2,3,6\}$), The Generalist consistently produces responses closest to consensus. Its $\Delta^{(t)}$ values show minimal drift, hovering near zero, while specialists show increasing adjustments.

\textbf{Final $\Delta^{(1000)}$ values} (selected dimensions):

\[
\Delta^{(1000)} = \begin{pmatrix}
-0.08 & -0.07 & -0.06 & 0.06 & 0.07 & 0.08 \\
0.01 & 0.01 & 0.00 & 0.00 & -0.01 & -0.01 \\
0.07 & 0.06 & 0.06 & -0.06 & -0.07 & -0.08
\end{pmatrix}
\]

\subsubsection{The Generalist Advantage}

The phenomenon emerges from the averaging mechanism in \cref{eq:R2}. When $k > 1$ dimensions are active:

\begin{enumerate}
    \item \textbf{Specialists suffer from high variance:} Strong performance in some dimensions is offset by weak performance in others.
    
    \item \textbf{The Generalist maintains consistency:} With uniform competencies, its averaged response is robust to dimension selection.
    
    \item \textbf{System consensus gravitates toward the middle:} The collective $G^{(t)}$ tends toward moderate values when perspectives have complementary strengths.
    
    \item \textbf{Learning amplifies the effect:} Through the feedback loop, specialists learn to moderate their responses ($\Delta^{(t)}$ adjustments), while the Generalist needs minimal adjustment.
\end{enumerate}

\subsubsection{Implications}

This ``promotion of the generalist'' demonstrates that in complex multi-dimensional evaluation tasks:
\begin{itemize}
    \item \textbf{Breadth can trump depth} when decisions require simultaneous consideration of diverse factors.
    \item \textbf{Systemic robustness} emerges from moderate, consistent competencies.
    \item \textbf{Specialized expertise} may be suboptimal when isolated from complementary skills. 
\end{itemize}

\subsection{Alternative Worked Example 2: The Halo Effect - Global Promotion of The Sole Expert}

\subsubsection{Setup: Singular Asymmetry}

To demonstrate how singular expertise propagates through coupled dimensions, we configure:

\begin{itemize}
    \item $n = 5$ perspectives: The Sole Expert ($\p_1$), Generalist A ($\p_2$), Generalist B ($\p_3$), Generalist C ($\p_4$), Generalist D ($\p_5$)
    \item $m = 19$ semantic dimensions: $\s_1,\s_2,\ldots,\s_{19}$
    \item $k = 3$ active dimensions per iteration
    \item $\alpha = 0.02$
\end{itemize}

The competency matrix $C \in \mathbb{R}^{5 \times 19}$ has a remarkable property:
\[
C_{i,j} = \begin{cases}
0.9 & \text{if } i = 1 \text{ and } j = 14 \text{ (The Sole Expert at } \s_{14}\text{)} \\
0.5 & \text{otherwise}
\end{cases}
\]

All perspectives are identical generalists except The Sole Expert has expertise in exactly one dimension: $\s_{14}$.

\subsubsection{Mechanism of Propagation}

Consider an iteration where $\s_{14}$ is active alongside two other dimensions.

\textbf{Iteration $t$:} Active dimensions $S^{(t)} = \{3, 14, 16\}$.

Signals: $x^{(t)}_{3} = 0.7$, $x^{(t)}_{14} = 0.3$, $x^{(t)}_{16} = 0.6$

For The Sole Expert ($\p_1$):
\begin{align*}
r_{1,3}^{(t)} &= 0.7 - 0.5 = 0.2 \\
r_{1,14}^{(t)} &= 0.3 - 0.9 = -0.6 \quad \text{(expert response)} \\
r_{1,16}^{(t)} &= 0.6 - 0.5 = 0.1
\end{align*}
Average: $R_1^{(t)} = \frac{0.2 + (-0.6) + 0.1}{3} = -0.1$

For any Generalist $\p_i$ (where $i \in \{2,3,4,5\}$):
\begin{align*}
r_{i,3}^{(t)} &= 0.7 - 0.5 = 0.2 \\
r_{i,14}^{(t)} &= 0.3 - 0.5 = -0.2 \\
r_{i,16}^{(t)} &= 0.6 - 0.5 = 0.1
\end{align*}
Average: $R_i^{(t)} = \frac{0.2 + (-0.2) + 0.1}{3} = 0.033$

\textbf{Critical observation:} The Sole Expert's strong response to $\s_{14}$ shifts its entire averaged response negative, distinguishing it from all other perspectives even for non-expert dimensions.

\subsubsection{The Halo Effect Emerges}

The consensus: $G^{(t)} = (-0.1) + 4 \times 0.033 = 0.032$

Normalized: $\bar{G}^{(t)} = \frac{0.032}{5} = 0.0064$

Learning signals:
\begin{align*}
d_1^{(t)} &= 0.0064 - (-0.1) = 0.1064 \quad \text{(The Sole Expert)} \\
d_i^{(t)} &= 0.0064 - 0.033 = -0.0266 \quad \text{for Generalists } i \in \{2,3,4,5\}
\end{align*}

\textbf{The key insight:} The Sole Expert receives a positive learning signal ($+0.1064$) while all Generalists receive negative signals ($-0.0266$). Crucially, due to coupled averaging, this adjustment applies to \textit{all three active dimensions}:

\[
\Delta_{1,j}^{(t+1)} = \Delta_{1,j}^{(t)} + \alpha \cdot 0.1064 \quad \text{for } j \in \{3, 14, 16\}
\]

The Sole Expert gains influence not just in $\s_{14}$ but also in $\s_3$ and $\s_{16}$---dimensions where it has no special expertise!

\subsubsection{Long-term Dynamics}

After 1000 iterations, the final $\Delta^{(1000)}$ matrix shows The Sole Expert has developed positive adjustments across all dimensions:

\[
\Delta_{1,j}^{(1000)} \approx \begin{cases}
0.08 \text{ to } 0.12 & \text{for all } j \in [19] \text{ (The Sole Expert)} \\
-0.02 \text{ to } -0.03 & \text{for all Generalists}
\end{cases}
\]

The system has spontaneously created a hierarchy where The Sole Expert becomes the de facto authority on \textit{everything}, despite only having true expertise in dimension $\s_{14}$.

\subsubsection{Analysis of the Halo Effect}

This phenomenon emerges from three interacting factors:

\begin{enumerate}
    \item \textbf{Asymmetric competency:} The Sole Expert's $C_{1,14} = 0.9$ creates consistently different responses when $\s_{14}$ is active.
    
    \item \textbf{Coupled dimensions:} Averaging across active dimensions means The Sole Expert's expertise signal spreads to co-active dimensions.
    
    \item \textbf{Feedback amplification:} Positive learning signals compound over iterations, gradually establishing The Sole Expert as influential across all dimensions.
\end{enumerate}

The mathematical mechanism: When $\s_{14}$ appears with probability $\frac{k}{m}$ per iteration, The Sole Expert accumulates positive adjustments that propagate to co-occurring dimensions. Over time, this creates a ``halo'' of influence extending far beyond the original expertise.

\subsubsection{Implications}

This emergent hierarchy from minimal initial asymmetry demonstrates:
\begin{itemize}
    \item \textbf{Authority spillover:} Domain-specific expertise can translate into broader systemic influence through structural coupling.
    \item \textbf{Spontaneous organization:} The system creates leadership structures without external designation. 
    \item \textbf{Fragility of equality:} Even small competency differences can cascade into large organizational asymmetries. 
\end{itemize}

\section{Discussion}

The properties of perspectives (or participants) in real application are typically not defined or known apriori but instead must be discovered by exercising them. While the competency matrix $C$ is presented here transparently in support of the mathematical formalism, the system itself does not strictly depend on $C$ but rather depends on measured responses of perspectives $R$. The matrix $C$ may be unknown, will typically be unknown, and the system will function.

\subsection{Use of Vector Averaging} Minary uses vector addition and averaging to represent the instantaneous states of the system. However, Minary is also a dynamical system where the ``head'' state represents an evolving process over time. Past states may be stored, creating a de facto time series, or may be discarded, creating a single integrated state representing the superposition of all past states.

\subsubsection{Part 1: Proof of Necessity}\leavevmode\\[1ex]
We claim that the Minary framework fails to exist without the averaging operation.

\textbf{The Mechanism of Consensus:} The core output of the system, the Consensus $G^{(t)}$, is formally defined as a summation of components. \cref{eq:G} states $G^{(t)}=\sum_{i=1}^{n}R_{i}^{(t)}$. Furthermore, the perspective's internal view is calculated by averaging across active dimensions: $R_{i}^{(t)}=\frac{1}{k}\sum_{j\in S^{(t)}}R_{i,j}^{(t)}$.

\textbf{The Convergence Target:} \cref{thm:convergence} proves that the system's stability depends entirely on converging to the mean of the competencies. The limit of the memory matrices $\Delta^{(t)}$ is defined by terms like $\overline{C}_{\cdot,j}$ (column averages) and $\overline{\overline{C}}$ (total average).

\textbf{Conclusion:} If you remove the averaging operation (linear superposition), the system cannot compute $G^{(t)}$ or resolve the deviations required for stability. Therefore, averaging is necessary.

\subsubsection{Part 2: Proof of Insufficiency}\leavevmode\\[1ex]
We claim that a pure averaging of input signals $x_{i,j}^{(t)}$ and competencies $C_{i,j}$ fails to replicate the behavior of the Minary framework.

\noindent\textbf{Evidence A:} The Memory Variable ($\Delta^{(t)}$)

\textbf{Averaging:} A standard weighted average is a function of current inputs: $\mathrm{Output}_t = f(\mathrm{Input}_t)$. It has no internal state.

\textbf{Minary:} The Minary output depends on previous learning. \cref{eq:R} defines the response as $R_{i,j}^{(t)}=r_{i,j}^{(t)}+\Delta_{i,j}^{(t-1)}$.

\textbf{The Proof:} At time $t$, two systems with the same input signal $x^{(t)}$ and the same competency matrix $C$ will produce \textit{different outputs} if their histories ($\Delta^{(t-1)}$) differ. Therefore, averaging inputs is insufficient to determine the state of the system.

\noindent\textbf{Evidence B:} Signal Cancellation (The Autopoietic Property)

\textbf{Averaging:} If one averages a signal $x$, the result is directly proportional to $x$. If $x$ doubles, the average doubles.

\textbf{Minary:} The update rule for the system identity ($\Delta$) is independent of the signal magnitude. \cref{eq:d_final} proves that the terms with stimuli $x_{j}^{(t)}$ cancel out during the learning step.

\textbf{The Proof:} The system learns the structure of the participants, not the content of the signal. A simple average of the signal would fail to capture this ``structural learning'' behavior.

\subsubsection{Summary}\leavevmode\\[1ex]

Averaging is a feed-forward controller \cite{Astrom,Wiener}. It takes inputs and pushes them to an output:
$$\mathrm{Output} = \text{Avg}(\mathrm{Inputs}).$$

Minary is a feed-back controller. It measures the error of the average and adds it to a memory bank:
$$\mathrm{Output} = \text{Avg}(\mathrm{Inputs}) + \int (\text{Error}) dt.$$

\subsection{Mapping the Features of Autopoiesis}

While Minary is a deterministic consequence of preconditions in the competency matrix $C$, this matrix is functionally hidden.

\begin{enumerate}
  \item  Here $C$ is provided solely for the purposes of the mathematical formalism. The system actually relies on $R$, while $C$ may be unknown.
  \item  If $C$ is unknown, the system still functions, so $(\Delta^{(t)})_{t\geq 0}$ reasonably represents the identity of the system, not $C$.
  \item  The matrix $C$ represents the identities of the parts while $\Delta^{(t)}$ represents the identity of the whole. 
  \item  The set $P$ of perspectives represents the closure because $C$ (and $R$) depends entirely on $P$.  
  \item  The signal $x^{(t)}$ represents the environment because it does not depend on $C$, $P$, or $(\Delta^{(t)})_{t\geq 0}$.
  \item  The consensus process $(G^{(t)})_{t\geq 0}$ represents what the system \textit{does}.
  \item  Finally, $P$ could be made dynamic if different perspectives participate in each iteration, and the system could still maintain a $(\Delta^{(t)})_{t\geq 0}$. This suggests the structure has \textit{an identity}.
\end{enumerate}

This provides a compelling argument for the satisfaction of Maturana and Varela’s criteria:

\begin{enumerate}
  \item Network of processes: Perspectives continuously transform inputs into responses that update the moving average $\Delta^{(t)}$.
  \item Regeneration: The moving average is continuously recreated through each iteration---it is not static but actively maintained.
  \item Concrete unity: The perspectives do not interact directly or affect each other but rather participate in the Minary protocol, which is what produces the collective identity that constitutes the system boundary.
  \item Operational openness: The system responds to environmental signals $x$.
  \item Organizational closure: But $x$ cancels out---the organization ($\Delta^{(t)}$) is determined entirely by internal dynamics.
  \item Turnover: Perspectives may swap, ``die'', or be ``born'', but the Minary identity $\Delta^{(t)}$ remains.
\end{enumerate}

\section{Logical Formalism for Autopoiesis}

We establish that Minary satisfies the criteria for autopoiesis through the following argument.

\textbf{Organizational Closure.} 
The matrices $\Delta^{(t)}$ exist and evolve over time. 
There are two possible sources for $\Delta^{(t)}$: external input $x$, or internal system dynamics. 
\cref{eq:d_final} demonstrates that $x$ cancels from the learning signal. 
Therefore, $(\Delta^{(t)})_{t\geq 0}$ is produced entirely through internal dynamics—the system produces itself.

\textbf{Operational Openness.} 
The consensus $G^{(t)}$ exists and responds to input $x^{(t)}$. 
However, $G^{(t)} \neq x^{(t)}$; the system transforms input rather than passing it through. 
Therefore, Minary is a process that responds to its environment.

\textbf{Conclusion.}
Minary is organizationally closed (self-producing) and operationally open (environmentally responsive). 
These are the defining characteristics of autopoiesis.

\section{Limitations}

It is worth mentioning that the discourse around autopoiesis has, from the very beginning, involved continuous debate between two ``levels'' \cite{Luisi,RazetoBarry}:

\begin{enumerate}
  \item Self-maintaining systems.
  \item Self-replicating systems.
\end{enumerate}

Minary as a primitive would fall more under level 1 as it is a system that can maintain itself. Whether self-replicating systems (level 2) could be constructed from self-maintaining primitives remains an open question, but a primitive is nonetheless necessary.

\textbf{On Self-Production in Computational Systems.}
A potential objection holds that while Minary exhibits organizational closure, it lacks genuine self-production—that updating matrix values is not equivalent to a cell producing proteins. We address this directly.

In biological autopoiesis, ``components'' are the physical constituents that realize the system's organization: membranes, enzymes, structural proteins. In computation, the analogous constituents are states---the values that realize the system's organizational identity. For Minary, this is the process $(\Delta^{(t)})_{t\geq 0}$. 

The matrix $\Delta^{(t)}$ is not static structure; it is continuously regenerated through each iteration's feedback dynamics. If $\Delta^{(t)}$ were frozen, the system would still process inputs, but it would cease to be the same system over time. The ongoing production of $\Delta^{(t)}$ is what maintains organizational identity.

The question ``does updating values count as production?'' reduces to: what else could computational self-production \cite{Fleischaker,BourgineStewart} mean? A quine produces its own source code, yet no one considers quines autopoietic \cite{Hofstadter}. They replicate but do not maintain. Self-replication without ongoing self-maintenance is not autopoiesis. Conversely, requiring production of code or hardware would define computational autopoiesis out of existence, since no running process produces its own substrate.

The coherent standard is that a system continuously produces the state that constitutes its organizational identity. Minary produces $\Delta^{(t)}$ not once, but on every iteration, through dynamics that are mathematically closed to external signal. This is self-maintenance, which is the more fundamental of the two levels historically associated with autopoiesis.

\section{On Usefulness}

While this article defines a foundational model using static competencies to formalize the underlying dynamics, the Minary framework supports significant architectural variation. In practical applications, the input signal may be deterministic rather than stochastic, or the perspectives themselves may evolve via the feedback loop. Such configurations establish a temporal trajectory for the EMA memory, wherein the input signal traces a path through a sparse topology.

Future implementations may also extend the domain to matrix-based signals or complex-valued components to increase expressivity. In this context, the matrix $\Delta^{(t)}$ functions not merely as an opaque internal state for driving consensus, but as a queryable manifold of the system's relative dispositions—effectively providing direct access to a ``subjective projection'' of the dataset.

The Minary primitive therefore serves as a flexible substrate for engineering applications. The fundamental invariant of the system is the instantaneous conservation of information, achieved through the distribution of perspective deviations from the global mean. Subject to this constraint, the framework allows for broad design latitude to suit specific implementation goals.

\section{Conclusion}

We believe Minary is a candidate for the first formally proven autopoietic computational primitive. We acknowledge the weight of this claim and hope that this article prompts discussion and new directions of inquiry.

The properties of autopoiesis: self-maintenance, coherence through feedback, and structural stability, suggest new possibilities for computational systems. Where traditional allopoietic architectures require external intervention to maintain function or adapt to new conditions, an autopoietic primitive could enable systems that are robust to component failure, adaptive without retraining, and capable of operating in environments without ground truth. The linearity and commutative properties of Minary's superposition additionally provide computational advantages: $O(n)$ complexity, natural parallelization, and suitability for distributed architectures. And perhaps most the intriguing property of all: Minary possesses uniquely relative learning dynamics that support what could be a form of a purely relative, subjective, identity.

\section*{Acknowledgments}
This work was supported by Autopoetic. Colin Defant was supported by a Benjamin Peirce Fellowship at Harvard University. 
The Minary computational primitive is patent-pending.

\end{document}